\documentclass[letterpaper, 10 pt, conference]{ieeeconf}
\IEEEoverridecommandlockouts
\let\labelindent\relax
\usepackage{enumitem}
\usepackage{balance}
\usepackage[font={small}]{caption}
\usepackage{subcaption}
\usepackage{array}
\usepackage{textcomp}
\usepackage{mathtools, nccmath}
\usepackage{graphicx}
\usepackage{amsfonts}
\usepackage{amsmath}
\usepackage{amssymb}
\usepackage{algorithm}
\usepackage{algorithmic}
\usepackage{hyperref}
\usepackage{tikz}
\usepackage{arydshln}
\usepackage{multirow}
\usepackage{bm}
\usepackage{epstopdf}
\usepackage{cite}
\usepackage{siunitx}

\usepackage{amsthm}

\newtheorem{theorem}{Theorem}

\newtheorem{remark}{Remark}
\theoremstyle{definition}
\newtheorem{definition}{Definition}
\newtheorem{problem}{Problem}

\title{\LARGE \bf
    Risk-Aware Adaptive Control Barrier Functions for Safe Control of Nonlinear Systems under Stochastic Uncertainty}

\author{Shuo Liu$^{1}$ and Calin A. Belta$^{2}$
\thanks{This work was supported in part by the NSF under grant IIS-2024606 at Boston University.}
\thanks{$^{1}$S. Liu is with the department of Mechanical Engineering, Boston
University, Brookline, MA, USA. 
        {\tt\small liushuo@bu.edu}}%
\thanks{$^{2}$C. Belta is with the Department of Electrical and Computer Engineering and the Department of Computer Science, University of Maryland, College Park, MD, USA 
        {\tt\small cbelta@umd.edu}}%
}

\begin{document} 
\maketitle

\begin{abstract}
This paper addresses the challenge of ensuring safety in stochastic control systems with high-relative-degree constraints, while maintaining feasibility and mitigating conservatism in risk evaluation. Control Barrier Functions (CBFs) provide an effective framework for enforcing safety constraints in nonlinear systems. However, existing methods struggle with feasibility issues and multi-step uncertainties. To address these challenges, we introduce Risk-aware Adaptive CBFs (RACBFs), which integrate Discrete-time Auxiliary-Variable adaptive CBFs (DAVCBFs) with coherent risk measures. DAVCBFs introduce auxiliary variables to improve the feasibility of the optimal control problem, while RACBFs incorporate risk-aware formulations to balance safety and risk evaluation performance. By extending discrete-time high-order CBF constraints over multiple steps, RACBFs effectively handle multi-step uncertainties that propagate through the system dynamics. We demonstrate the effectiveness of our approach on a stochastic unicycle system, showing that RACBFs maintain safety and feasibility while reducing unnecessary conservatism compared to standard robust formulations of discrete-time CBF methods.
\end{abstract}

\section{Introduction}
\label{sec:Introduction}

Constrained optimal control problems with safety requirements are fundamental in emerging safety-critical autonomous and cyber-physical systems, which are often modeled using stochastic systems. These systems must operate safely in unpredictable environments while interacting with humans, whose behavior is difficult to predict. To address risk, systems science and control engineering offer a rich history of methods for handling uncertainty, broadly categorized into worst-case, risk-neutral, and risk-aware approaches \cite{wang2022risk}. The risk-aware paradigm offers a balanced approach between worst-case and risk-neutral methods. While the worst-case paradigm ensures safety by considering the most severe violations \cite{zhou1998essentials}, it often leads to overly conservative and impractical solutions. On the other hand, the risk-neutral paradigm \cite{bertsekas1996stochastic} optimizes average performance but fails to account for rare but critical failures. Risk-aware methods address these limitations by considering outcomes beyond the average while avoiding excessive conservatism \cite{howard1972risk}, leading to more practical and adaptive decision-making in safety-critical systems.

Recent studies have demonstrated that stabilizing an affine control system while ensuring safety constraints and control limitations can be achieved by unifying Control Barrier Functions (CBFs) with Control Lyapunov Functions (CLFs), forming a sequence of single-step optimization programs \cite{ames2014control, ames2016control}. When using a quadratic cost, these optimizations become quadratic programs (QP), enabling real-time implementation \cite{ames2014control}. Extensions of safety-critical control with CBFs have introduced formulations for high relative degree constraints and adaptive control \cite{nguyen2016exponential, xiao2021high, liu2023auxiliary, liu2024auxiliary, liu2025auxiliary}. To address safety in discrete-time systems, discrete-time CBFs (DCBFs) were introduced in \cite{agrawal2017discrete}, while discrete-time high-order CBFs (DHOCBFs) were later proposed in \cite{xiong2022discrete} to handle high relative degree constraints. All the CBFs mentioned above are for
deterministic systems.

Various approaches have been proposed in the literature to enforce safety in stochastic systems using CBFs. One approach to incorporating robust control in CBFs is to enforce safety by considering worst-case bounds on random variables and uncertainties \cite{lindemann2019control, cohen2022high}. However, as previously mentioned, this often results in overly conservative behavior, keeping the system state far from the safe set boundary. A second approach ensures probabilistic satisfaction of safety constraints, meaning the safety condition holds with a probability exceeding a given threshold. This has been incorporated into the CBF framework using stochastic (Itô) differential equations \cite{clark2021control, santoyo2021barrier} or chance constraints \cite{wang2021chance}. Another approach uses extended class $\kappa$ functions instead of class $\kappa$ functions when defining a CBF \cite{jankovic2018robust, tan2021high}, ensuring that the system state is always stabilized to the safe set whenever the safety constraint is violated. However, these methods are unsuitable for scenarios where a single failure could result in catastrophic consequences.
The recently proposed risk-aware DCBF approach \cite{singletary2022safe, akella2024risk} provides a middle ground between worst-case robustness and probabilistic safety constraint satisfaction by explicitly considering both the likelihood and severity of undesirable outcomes. However, this approach is applicable exclusively to safety constraints that possess a relative degree of one. This means that if a safety constraint has a high relative degree, the uncertainties will propagate dynamically 
over multiple steps corresponding to the relative degree, making them difficult to compensate for with low-order constraints and potentially compromising safety. Moreover, the optimization problems in the aforementioned studies often become infeasible due to conflicts between stochastic safety constraints and control bounds. Additionally, these works do not address improving feasibility under such constraints.

In order to address the challenges of conservativeness, feasibility and multi-step uncertainties when enforcing safety using DCBFs for stochastic systems, this paper contributes a general form of risk-aware adaptive DHOCBF. Specifically, we define discrete-time auxiliary-variable adaptive CBFs (DAVCBFs), which introduce auxiliary variables that multiply each DCBF and incorporate adaptive dynamics for these variables to construct the corresponding DCBF constraint. This approach allows for the relaxation of the DAVCBF constraint, improving the feasibility of the DCBF-based optimization problem. Meanwhile, DAVCBFs ensure safety for stochastic constraints with arbitrary relative degree by effectively handling multi-step uncertainties that propagate over multiple stages. To address the conservativeness of risk evaluation, we propose risk-aware adaptive (discrete-time) CBFs (RACBFs), which are compositions of DAVCBFs and coherent risk measures. The existence of RACBFs ensures safety in a risk-sensitive manner, avoiding the assumption of the absolute worst case when it is highly improbable, thereby achieving a better trade-off between safety and risk assessment performance. We demonstrate the effectiveness of the proposed method through its application to a unicycle system in a safety-critical scenario.

\section{Definitions and Preliminaries}
\label{sec:Preliminaries}
\subsection{Discrete-Time High-Order CBFs}
\label{subsec:DHOCBFs}
We consider a discrete-time control system in the form  
\begin{equation}
\label{eq:discrete-dynamics}
\mathbf{x}_{t+1}=f(\mathbf{x}_{t},\mathbf{u}_{t}),
\end{equation}
where $\mathbf{x}_{t} \in \mathcal X \subset \mathbb{R}^{n}$ represents the state of the system at time step $t\in\mathbb{N}, \mathbf{u}_{t}\in \mathcal U \subset \mathbb{R}^{q}$ is the control input, and the function $f: \mathbb{R}^{n}\times \mathbb{R}^{q}\to \mathbb{R}^{n}$ is locally Lipschitz.
\begin{definition}[Relative degree~\cite{sun2003initial}]
\label{def:relative-degree}
Let the output mapping be \( h: \mathbb{R}^n \rightarrow \mathbb{R} \). The output $y_{t}=h(\mathbf{x}_{t})$ of system \eqref{eq:discrete-dynamics} is said to have relative degree $m$ if
\begin{equation}
\begin{split}
&y_{t+i}=h(\bar{f}_{i-1}(f(\mathbf{x}_{t},\mathbf{u}_{t}))), \ i \in \{1,2,\dots,m\},\\
 \text{s.t.} & \ \frac{\partial y_{t+m}}{\partial \mathbf{u}_{t}} \ne \textbf{0}_{q}, \frac{\partial y_{t+i}}{\partial \mathbf{u}_{t}}= \textbf{0}_{q},  \ i \in \{1,2,\dots,m-1\},
\end{split}
\end{equation}
i.e., $m$ is the number of steps (delay) in the output $y_{t}$ in order for any component of the control input $\mathbf{u}_{t}$ to explicitly appear ($\textbf{0}_{q}$ is the zero vector of dimension $q$). 
\end{definition}

In the above definition, we use $\bar{f}(\mathbf{x})$ to denote the uncontrolled state dynamics $f(\mathbf{x}, 0)$. The subscript $i$ of function $\bar{f}(\cdot)$ denotes the $i$-times recursive compositions of $\bar{f}(\cdot)$, i.e.,  $\bar{f}_{i}(\mathbf{x})=\underset{i\text{-times}~~~~~~~~~~~~~}{\underbrace{\bar{f}(\bar{f}(\dots,\bar{f}}(\bar{f}_{0}(\mathbf{x}))))}$ with $\bar{f}_{0}(\mathbf{x})=\mathbf{x}$.

We assume that $h(\mathbf{x})$ has relative degree $m$ with respect to system~(\ref{eq:discrete-dynamics}) based on Def~\ref{def:relative-degree}.  
Starting with $\psi_{0}(\mathbf{x}_{t}) \coloneqq h(\mathbf{x}_{t})$, we define a sequence of discrete-time functions $\psi_{i}: \mathbb{R}^{n} \to \mathbb{R}$ for $i=1,\dots,m$ as:
\begin{equation}
\label{eq:high-order-discrete-CBFs}
\psi_{i}(\mathbf{x}_{t})\coloneqq \bigtriangleup \psi_{i-1}(\mathbf{x}_{t})+\alpha_{i}(\psi_{i-1}(\mathbf{x}_{t})), 
\end{equation}
where $\bigtriangleup \psi_{i-1}(\mathbf{x}_{t})\coloneqq \psi_{i-1}(\mathbf{x}_{t+1})-\psi_{i-1}(\mathbf{x}_{t})$, and $\alpha_{i}(\cdot)$ denotes the $i^{th}$ class $\kappa$ function which satisfies $\alpha_{i}(\psi_{i-1}(\mathbf{x}_{t}))\le \psi_{i-1}(\mathbf{x}_{t})$ for $i=1,\ldots, m$.
A sequence of sets $\mathcal {C}_{i}$ is defined based on \eqref{eq:high-order-discrete-CBFs} as
\begin{equation}
\label{eq:high-order-safety-sets}
\mathcal {C}_{i}\coloneqq \{\mathbf{x}\in \mathbb{R}^{n}:\psi_{i}(\mathbf{x})\ge 0\}, \ i \in\{0,\ldots,m-1\}.
\end{equation}

\begin{definition}[DHOCBF~\cite{xiong2022discrete}]
\label{def:high-order-discrete-CBFs}
Let $\psi_{i}(\mathbf{x}), \ i\in \{1,\dots,m\}$ be defined by \eqref{eq:high-order-discrete-CBFs} and $\mathcal {C}_{i},\ i\in \{0,\dots,m-1\}$ be defined by \eqref{eq:high-order-safety-sets}. A function $h:\mathbb{R}^{n}\to\mathbb{R}$ is a Discrete-Time High-Order Control Barrier Function (DHOCBF) with relative degree $m$ for system \eqref{eq:discrete-dynamics} if there exist $\psi_{m}(\mathbf{x})$ and $\mathcal {C}_{i}$ such that
\begin{equation}
\label{eq:highest-order-CBF}
\psi_{m}(\mathbf{x}_{t})\ge 0, \ \forall x_{t}\in \mathcal{C}_{0}\cap \dots \cap \mathcal {C}_{m-1}, t\in \mathbb{N}_{\geq 0}.
\end{equation}
\end{definition}

\begin{theorem}[Safety Guarantee \cite{xiong2022discrete}]
\label{thm:forward-invariance}
Given a DHOCBF $h(\mathbf{x})$ from Def. \ref{def:high-order-discrete-CBFs} with corresponding sets $\mathcal{C}_{0}, \dots,\mathcal {C}_{m-1}$ defined by \eqref{eq:high-order-safety-sets}, if $\mathbf{x}_{0} \in \mathcal {C}_{0}\cap \dots \cap \mathcal {C}_{m-1},$ then any Lipschitz controller $\mathbf{u}_{t}$ that satisfies the constraint in \eqref{eq:highest-order-CBF}, $\forall t\ge 0$ renders $\mathcal {C}_{0}\cap \dots \cap \mathcal {C}_{m-1}$ forward invariant for system \eqref{eq:discrete-dynamics}, $i.e., \mathbf{x}_{t} \in \mathcal {C}_{0}\cap \dots \cap \mathcal {C}_{m-1}, \forall t\ge 0.$
\end{theorem}
We can simply define an $i^{th}$ order DCBF $\psi_{i}(\mathbf{x})$ in \eqref{eq:high-order-discrete-CBFs} as
\begin{equation}
\label{eq:simple-high-order-discrete-CBFs}
\psi_{i}(\mathbf{x}_{t})\coloneqq \bigtriangleup \psi_{i-1}(\mathbf{x}_{t})+\gamma_{i}\psi_{i-1}(\mathbf{x}_{t}),
\end{equation}
where $\alpha_{i}(\cdot)$ is a linear function and $0<\gamma_{i}\le 1, i\in \{1,\dots,m\}$.
\subsection{Coherent Risk Measure}
\label{subsec: Coherent Risk Measure}
The aim of this section is to introduce conditional risk measures as a foundation for defining risk-based CBFs in Sec. \ref{sec:RACBFs}. Consider a probability space $(\Omega, \mathcal{F}, \mathbb{P})$, a filtration $\mathcal{F}_0 \subseteq \dots \subseteq \mathcal{F}_N \subseteq \mathcal{F}$, and an adapted sequence of random variables $h_t, t = 0, \dots, N$, where $N \in \mathbb{N}_{\geq 0} \cup \{\infty\}$. For $t = 0, \dots, N$, we further define the spaces $\mathcal{H}_t = L_p(\Omega, \mathcal{F}_t, \mathbb{P}), p \in [0, \infty)$, $\mathcal{H}_{t:N} = \mathcal{H}_t \times \dots \times \mathcal{H}_N$, and $\mathcal{H} = \mathcal{H}_0 \times \mathcal{H}_1 \times \dots$. We assume that the sequence $\boldsymbol{h} \in \mathcal{H}$ is almost surely bounded (with probability zero exceptions), i.e., $\operatorname{ess} \sup_t |h_t(\cdot)| < \infty$.
To evaluate the risk of the sub-sequence $h_t, \dots, h_N$ from the perspective of stage $t$, we require the following definitions.
\begin{definition}[Conditional Risk Measure]
A mapping $\rho_{t:N} : \mathcal{H}_{t:N} \to \mathcal{H}_t$, where $0 \leq t \leq N$, is called a conditional risk measure if it satisfies the following monotonicity property:
\begin{equation}
    \rho_{t:N}(h) \leq \rho_{t:N}(h'), \quad \forall h, h' \in \mathcal{H}_{t:N} \text{ such that } h \leq h'.
\end{equation}
A dynamic risk measure is a sequence of conditional risk measures $\rho_{t:N} : \mathcal{H}_{t:N} \to \mathcal{H}_t$, for $t = 0, \dots, N$.
\end{definition}
A fundamental property of dynamic risk measures is their consistency over time \cite{ruszczynski2010risk}. This means that if a risk evaluation at some future time $t_\tau$ considers $h$ to be as good as $h'$, and if $h$ and $h'$ are identical between times $t_k$ and $t_\tau$ ($k<\tau$), then from the perspective of time $t_k$, $h$ should not be considered worse than $h'$. For a time-consistent risk measure, the one-step conditional risk measure $\rho_t$ is defined as a mapping:
\begin{equation}
    \rho_t(h_t) = \rho_{t-1: t}(0, h_t),
\end{equation}
and for all $t = 1, \dots, N$, we obtain:
\begin{equation}
\label{eq:rho-measure}
\begin{split}
    \rho_{t:N} (h_t, \dots, h_N) = \rho_t ( h_t + \rho_{t+1} ( h_{t+1} + \rho_{t+2} ( h_{t+2} + \dots \\
    + \rho_{N-1} ( h_{N-1} + \rho_N (h_N) ) \dots ))).
\end{split}
\end{equation}
Thus, time-consistent risk measures are fully determined by their one-step conditional risk measures $\rho_t$, for $t = 0, \dots, N-1$. In particular, when $t = 0$, this defines a risk measure for the entire sequence $\boldsymbol{h} \in \mathcal{H}_{0:N}$. This leads to the notion of a coherent risk measure.
\begin{definition}[Coherent Risk Measure]
\label{def:Coherent Risk Measure}
The one-step conditional risk measures $\rho_t : \mathcal{H}_{t} \to \mathcal{H}_{t-1}$, $t = 1, \dots, N$, as in \eqref{eq:rho-measure}, are called coherent risk measures if they satisfy the following properties:
\begin{itemize}
    \item \textbf{Convexity}: $\rho_t(\lambda h + (1-\lambda) h') \leq \lambda \rho_t(h) + (1-\lambda) \rho_t(h')$, \quad for all $\lambda \in (0,1)$ and all $h, h' \in \mathcal{H}_t$.
    \item \textbf{Monotonicity}: If $h \leq h'$, then $\rho_t(h) \leq \rho_t(h')$, for all $h, h' \in \mathcal{H}_t$.
    \item \textbf{Translational Invariance}: $\rho_t(h + c) = c + \rho_t(h)$, for all $h \in \mathcal{H}_{t}$ and $c \in \mathbb{R}$.
    \item \textbf{Positive Homogeneity}: $\rho_t(\beta h) = \beta \rho_t(h)$, for all $h \in \mathcal{H}_t$ and $\beta \geq 0$.
    \item \textbf{Normalization}: $\rho_t(0) = 0$.
\end{itemize}
\end{definition}
All risk measures discussed in this paper are time-consistent and coherent. Below, we briefly review conditional value-at-risk as one key example of coherent risk measures.

\textbf{Conditional Value-at-Risk (CVaR)}
Let $h \in \mathcal{H}$ be a stochastic variable where higher values indicate safer performance. Since we aim to prevent $h$ from becoming negative (as seen in Eqn. \eqref{eq:high-order-safety-sets}), we treat negative values of $h$ as a loss and measure left-tail risk using Conditional Value-at-Risk (CVaR) or Value-at-Risk (VaR) at the lower quantile. For a confidence level $\beta \in (0,1)$, VaR is defined as:
\begin{equation}
    \text{VaR}_\beta(h) = \sup \{ \zeta \in \mathbb{R} \mid \mathbb{P}(h \leq \zeta) \leq \beta \}.
\end{equation}
However, VaR is numerically unstable for non-normal stochastic variables, and optimization problems involving VaR become intractable in high-dimensional settings. Moreover, VaR ignores the magnitude of values below the $\beta$-quantile \cite{rockafellar2000optimization}. In contrast, CVaR addresses these issues by considering the expected loss in the $\beta$-tail. The standard CVaR with confidence level $\beta$ is given by:
\begin{equation}
\label{eq:standard CVaR}
    \text{CVaR}_\beta(h) = \mathbb{E} [h \mid h \leq \text{VaR}_\beta(h)].
\end{equation}
An alternative formulation introduced in \cite{rockafellar2000optimization} is:
\begin{equation}
\label{eq:better CVaR}
    \text{CVaR}_\beta(h) = - \inf_{\zeta \in \mathbb{R}} \mathbb{E} \left[ \zeta + \frac{( - h - \zeta )^+}{\beta} \right].
\end{equation}
This formulation is used in this paper instead of \eqref{eq:standard CVaR} because it is more tractable for optimization and ensures time consistency in sequential decision-making. The standard CVaR formula \eqref{eq:standard CVaR} is non-convex and difficult to optimize, whereas Eq. \eqref{eq:better CVaR} reformulates CVaR in a convex and computationally efficient way. Additionally, the standard CVaR definition does not extend naturally to multi-stage problems, making it unsuitable for risk-aware control applications. Equation \eqref{eq:better CVaR} provides a flexible framework that integrates well with dynamic models, making it more practical for real-world safety-critical systems. A confidence level $\beta \to 1$ corresponds to a risk-neutral case, where $\text{CVaR}_1(h) = \mathbb{E}[h]$. Conversely, $\beta \to 0$ results in a risk-averse measure, where $\text{CVaR}_0(h) = \text{VaR}_0(h) = \operatorname{ess} \inf(h)$ \cite{rockafellar2002conditional}. 

\section{Problem Formulation and Approach}
\label{sec:Problem Formulation and Approach}
Consider a discrete-time stochastic control system:
\begin{equation}
\label{eq:disturbed system}
    \mathbf{x}_{t+1} = f(\mathbf{x}_t, \mathbf{u}_t, \mathbf{w}_t),
\end{equation}
where $t \in \mathbb{N}_{\geq 0}$ denotes the time index,  
$\mathbf{x}_t \in \mathcal{X} \subset \mathbb{R}^n$ is the state,  
$\mathbf{u}_t \in \mathcal{U} \subset \mathbb{R}^m$ is the control input,  
$\mathbf{w}_t \in \mathcal{W}$ represents the stochastic uncertainty/disturbance,  
and the function $f: \mathbb{R}^n \times \mathcal{U} \times \mathcal{W} \to \mathbb{R}^n$.  
We assume that the initial condition $x_0$ is deterministic and that $|\mathcal{W}|$ is finite, i.e., $\mathcal{W} = \{\mathbf{w}_1, \dots, \mathbf{w}_{|\mathcal{W}|}\}$.
At every time step $t$, for a state-control pair $(\mathbf{x}_t, \mathbf{u}_t)$, the process disturbance $\mathbf{w}_t$ is drawn from the set $\mathcal{W}$ according to the probability mass function  
$p(\mathbf{w}) = [p(\mathbf{w}_1), \dots, p(\mathbf{w}_{|\mathcal{W}|})]^T$, where $p(\mathbf{w}_i) := \mathbb{P}(\mathbf{w}_t = \mathbf{w}_i), i = 1, 2, \dots, |\mathcal{W}|$. The probability mass function for the process disturbance remains constant over time and is independent of both the process history and $(\mathbf{x}_t, \mathbf{u}_t)$.
\textbf{Objective:} (Minimizing cost) We consider the following cost for system \eqref{eq:disturbed system}:  
\begin{equation}
\label{eq:cost-function-1}
\begin{split}
J(\mathbf{u}_{t},\mathbf{x}(\mathbf{u}_t))=\mathbf{u}_{t}^{T}Q\mathbf{u}_{t}+(\mathbf{x}(\mathbf{u}_t)-\mathbf{x}_{e})^{T}R(\mathbf{x}(\mathbf{u}_t)-\mathbf{x}_{e}),
\end{split}
\end{equation}
where $\mathbf{x}_{e}$ denotes a reference state, $0\le t\le t_T, t \in \mathbb{N}$, and $Q,R$ are weight matrices with positive diagonal entries.

\textbf{Safety Requirement:} System \eqref{eq:disturbed system} should always satisfy a safety requirement of the form: 
\begin{equation}
\label{eq:Safety constraint}
h(\mathbf{x}_{t})\ge 0, 0\le t\le t_T, t \in \mathbb{N},
\end{equation}
where $h:\mathbb{R}^{n}\to\mathbb{R}$ is assumed to be a continuous function and has relative degree $m\in \mathbb{N}$ with respect to system \eqref{eq:disturbed system}.

\textbf{Control Limitation Requirement:} The controller $\mathbf{u}_{t}$ should always satisfy $\mathbf{u}_{t}\in \mathcal U \subset \mathbb{R}^{q}$ for all $0\le t\le t_T$.

A control policy is \textbf{feasible} if all constraints derived from safety requirement and control limitation requirement are satisfied and mutually non-conflicting for all $0\le t\le t_T$. In this paper, we consider the following problem:

\begin{problem}
\label{prob:opt-prob}
Find a feasible control policy for system \eqref{eq:disturbed system} and cost \eqref{eq:cost-function-1} is minimized.
\end{problem}

 \textbf{Approach:} Our method for solving Problem \ref{prob:opt-prob} is based on the coherent risk measure introduced at the Sec. \ref{subsec: Coherent Risk Measure}. We define Discrete-Time Auxiliary-Variable Adaptive CBFs (DAVCBFs) to enforce the safety constraint \eqref{eq:Safety constraint} by introducing auxiliary variables that construct an adaptive DCBF constraint. These auxiliary variables are controlled by auxiliary inputs, which act as relaxation variables to improve the feasibility of the DCBF-based optimization problem under control bounds $\mathbf{u}_{t}\in \mathcal U \subset \mathbb{R}^{q}$. We show that DHOCBF is a special case of DAVCBF, meaning DAVCBFs can enforce safety for stochastic constraints with a relative degree higher than one.
To assess and enforce risk-sensitive safety, we introduce Risk-Aware Adaptive CBFs (RACBFs), which incorporate risk measures into the DCBF framework to account for stochastic uncertainties while enforcing safety constraints.
 
\section{Risk-Aware Adaptive Control Barrier Functions}
\label{sec:RACBFs}
\subsection{Enhancing Optimization Feasibility with DAVCBFs}
Motivated by the auxiliary-variable adaptive CBFs in \cite{liu2023auxiliary}, given a function $h:\mathbb{R}^{n}\to\mathbb{R}$ with relative degree $m$ for system \eqref{eq:discrete-dynamics}, and a time-varying vector of auxiliary variables $\boldsymbol{a}_t \coloneqq [a_{1,t},\dots,a_{m,t}]^{T}$ with positive components, we place one auxiliary variable in front of each function $\psi_{i}(\mathbf{x})$ in \eqref{eq:high-order-discrete-CBFs}. If each auxiliary variable is ensured to always be positive as $a_{i,t} > 0, \quad i \in \{1, \dots, m\}$, then multiplying the auxiliary variable $a_{i,t}$ in front of $\psi_{i-1}(\mathbf{x})$ is equivalent to enforcing $\psi_{i-1}(\mathbf{x}) \geq 0$. Therefore, the forward invariance of the intersection of a sequence of sets $\mathcal{C}_{i}$ defined by \eqref{eq:high-order-safety-sets} can still be rendered with auxiliary variables. To ensure that $a_{i,t}$ remains positive, we define a DHOCBF for each $a_{i,t}$ and introduce auxiliary systems containing auxiliary states $\boldsymbol{\pi}_{i,t}$ and inputs $\nu_{i,t}$. Through these systems, we extend each DHOCBF to the desired relative degree, similar to how $h(\mathbf{x})$ has relative degree $m$ with respect to the dynamics \eqref{eq:discrete-dynamics}. Consider $m$ auxiliary systems in the form:
\begin{equation}
\label{eq:virtual-system}
\boldsymbol{\pi}_{i,t+1}=F_{i}(\boldsymbol{\pi}_{i,t}, \nu_{i,t}), i \in \{1,...,m\},
\end{equation}
where $\boldsymbol{\pi}_{i, t}\coloneqq [\pi_{i,1,t},\dots,\pi_{i,m+1-i,t}]^{T}\in \mathbb{R}^{m+1-i}$ denotes an auxiliary state with $\pi_{i,j,t}\in \mathbb{R}, j \in \{1,...,m+1-i\}.$ $\nu_{i,t}\in \mathbb{R}$ denotes an auxiliary input for \eqref{eq:virtual-system}, $F_{i}:\mathbb{R}^{m+1-i}\times \mathbb{R}\to \mathbb{R}^{m+1-i}$ 
is locally Lipschitz. For simplicity, we just build up the connection between an auxiliary variable and the system as $a_{i,t}=\pi_{i,1,t}, \pi_{i,1,t+1}-\pi_{i,1,t}=\pi_{i,2,t},\dots,\pi_{i,m-i,t+1}-\pi_{i,m-i,t}=\pi_{i,m+1-i,t}$ and make $\pi_{i,m+1-i,t+1}-\pi_{i,m+1-i,t}=\nu_{i,t}$, then we can define many specific DHOCBFs $h_{i}$ to enable $a_{i,t}$ to be positive. 
Given a function $h_{i}:\mathbb{R}^{m+1-i}\to\mathbb{R},$ we can define a sequence of functions $\varphi_{i,j}:\mathbb{R}^{m+1-i}\to\mathbb{R}, i \in\{1,...,m\}, j \in\{1,...,m+1-i\}:$
\begin{equation}
\label{eq:virtual-DHOCBFs}
\varphi_{i,j}(\boldsymbol{\pi}_{i, t})\coloneqq \bigtriangleup \varphi_{i,j-1}(\boldsymbol{\pi}_{i, t})+\alpha_{i,j}(\varphi_{i,j-1}(\boldsymbol{\pi}_{i, t})),
\end{equation}
where $\varphi_{i,0}(\boldsymbol{\pi}_{i,t})\coloneqq h_{i}(\boldsymbol{\pi}_{i,t})$, $\bigtriangleup \varphi_{i,j-1}(\boldsymbol{\pi}_{i,t})\coloneqq \varphi_{i,j-1}(\boldsymbol{\pi}_{i,t+1})-\varphi_{i,j-1}(\boldsymbol{\pi}_{i,t})$, $\alpha_{i,j}(\cdot)$ are class $\kappa$ functions that satisfy $\alpha_{i,j}(\varphi_{i,j-1}(\boldsymbol{\pi}_{i,t}))\le \varphi_{i,j-1}(\boldsymbol{\pi}_{i,t})$. Sets $\mathcal{B}_{i,j}$ are defined as
\begin{equation}
\label{eq:virtual-sets}
\mathcal B_{i,j}\coloneqq \{\boldsymbol{\pi}_{i,t}\in\mathbb{R}^{m+1-i}:\varphi_{i,j}(\boldsymbol{\pi}_{i,t})>0\}, \ j\in \{0,...,m-i\}. 
\end{equation}
Let $\varphi_{i,j}(\boldsymbol{\pi}_{i}),\ j\in \{1,...,m+1-i\}$ and $\mathcal B_{i,j},\ j\in \{0,...,m-i\}$ be defined by \eqref{eq:virtual-DHOCBFs} and \eqref{eq:virtual-sets} respectively. By Def. \ref{def:high-order-discrete-CBFs}, a function $h_{i}:\mathbb{R}^{m+1-i}\to\mathbb{R}$ is a DHOCBF with relative degree $m+1-i$ for system \eqref{eq:virtual-system} if there exist class $\kappa$ functions $\alpha_{i,j},\ j\in \{1,...,m+1-i\}$ as in \eqref{eq:virtual-DHOCBFs} such that
\begin{small}
\begin{equation}
\label{eq:highest-SHOCBF}
\begin{split}
\sup_{\nu_{i,t}\in \mathbb{R}}[\sum_{k=0}^{j_{m}}[(-1)^{j_{m}-k}\binom{j_{m}}{j_{m}-k}\varphi_{i,0}(\boldsymbol{\pi}_{i,t+k})]+
\sum_{n=1}^{j_{m}}[\sum_{k=0}^{j_{m}-n}[\\
(-1)^{j_{m}-n-k}\binom{j_{m}-n}{k}\alpha_{i,n}(\varphi_{i,n-1}(\boldsymbol{\pi}_{i, t+k}))]]] \ge \epsilon,
\end{split}
\end{equation}
\end{small}
where $\epsilon$ is a positive constant which can be infinitely small, the auxiliary input $\nu_{i,t}$ will appear in $\varphi_{i,0}(\boldsymbol{\pi}_{i,t+j_{m}})$ and $j_{m}=m+1-i$. 
\begin{remark}
\label{rem:safety-guarantee-2}
Eqn. \eqref{eq:highest-SHOCBF} and $\varphi_{i,m+1-i}(\boldsymbol{\pi}_{i, t})\ge \epsilon$ are equivalent. If this inequality constraint is satisfied, we have $\bigtriangleup \varphi_{i,m-i}(\boldsymbol{\pi}_{i, t})+\alpha_{i,m+1-i}(\varphi_{i,m-i}(\boldsymbol{\pi}_{i, t}))\ge \epsilon$. This can be rewritten as
\begin{equation}
\label{eq:virtual-sets1}
\varphi_{i,m-i}(\boldsymbol{\pi}_{i, t+1})-\varphi_{i,m-i}(\boldsymbol{\pi}_{i, t})+\alpha_{i,m+1-i}(\varphi_{i,m-i}(\boldsymbol{\pi}_{i, t}))\ge \epsilon.
\end{equation}
Since $\delta_{i,m-i,t} =\varphi_{i,m-i}(\boldsymbol{\pi}_{i, t})-\alpha_{i,m+1-i}(\varphi_{i,m-i}(\boldsymbol{\pi}_{i, t}))+\epsilon>0$, we have $\varphi_{i,m-i}(\boldsymbol{\pi}_{i, t+1})\ge \delta_{i,m-i,t} >0$. Recursively we have $\varphi_{i,m-i}(\boldsymbol{\pi}_{i, t+k})\ge \delta_{i,m-i,t+k-1} >0$ where $k \in \{1,2,\dots\}$. Similarly, we have 
\begin{equation}
\begin{split}
\label{eq:virtual-sets2}
\varphi_{i,m-1-i}(\boldsymbol{\pi}_{i, t+k})-\varphi_{i,m-1-i}(\boldsymbol{\pi}_{i, t+k-1})+\\
\alpha_{i,m-i}(\varphi_{i,m-1-i}(\boldsymbol{\pi}_{i, t+k-1}))>0,
\end{split}
\end{equation}
then we have $\varphi_{i,m-1-i}(\boldsymbol{\pi}_{i, t+k})>0$, $k \in \{1,2,\dots\}$. Repeatedly we have $\varphi_{i,0}(\boldsymbol{\pi}_{i, t+k})>0$, $k \in \{1,2,\dots\}$, therefore the sets $\mathcal {B}_{i,0},\dots,\mathcal {B}_{i,m-i}$ are forward invariant.
\end{remark}
For simplicity, we can make $h_{i}(\boldsymbol{\pi}_{i,t})=\pi_{i,1,t}=a_{i,t}.$ Based on Rem. \ref{rem:safety-guarantee-2}, each $a_{i,t}$ will be positive.
The remaining question is how to define an adaptive DHOCBF to guarantee $h(\mathbf{x}) \geq 0$ with the assistance of auxiliary variables. Let $\boldsymbol{\Pi}_t \coloneqq [\boldsymbol{\pi}_{1,t},\dots,\boldsymbol{\pi}_{m,t}]^{T}, \boldsymbol{\nu}_t \coloneqq [\nu_{1,t},\dots,\nu_{m,t}]^{T}$ denote the auxiliary states and control inputs of system \eqref{eq:virtual-system}. We can define a sequence of functions as follows:

\begin{small}
\begin{equation}
\label{eq:DAVCBF-sequence}
\begin{split}
&\psi_{0}(\mathbf{x}_t,\boldsymbol{\Pi}_t)\coloneqq a_{1,t}h(\mathbf{x}_t),\\
&\psi_{i}(\mathbf{x}_t,\boldsymbol{\Pi}_t)\coloneqq a_{i+1,t}(\bigtriangleup \psi_{i-1}(\mathbf{x}_t,\boldsymbol{\Pi}_t)+\alpha_{i}(\psi_{i-1}(\mathbf{x}_t,\boldsymbol{\Pi}_t))),
\end{split}
\end{equation}
\end{small}
where $i \in \{1,...,m-1\}$, $\bigtriangleup \psi_{i-1}(\mathbf{x}_{t}, \boldsymbol{\Pi}_t)\coloneqq \psi_{i-1}(\mathbf{x}_{t+1},\boldsymbol{\Pi}_{t+1})-\psi_{i-1}(\mathbf{x}_{t},\boldsymbol{\Pi}_t)$. $\psi_{m}(\mathbf{x}_t, \boldsymbol{\Pi}_t)\coloneqq \bigtriangleup \psi_{m-1}(\mathbf{x}_t,\boldsymbol{\Pi}_t)+\alpha_{m}(\psi_{m-1}(\mathbf{x}_t,\boldsymbol{\Pi}_t)).$ We further define a sequence of sets $\mathcal{C}_{i}$ associated with \eqref{eq:DAVCBF-sequence} in the form 
\begin{equation}
\label{eq:DAVCBF-set}
\begin{split}
\mathcal C_{i}\coloneqq \{(\mathbf{x}_t,\boldsymbol{\Pi}_t) \in \mathbb{R}^{n} \times \mathbb{R}^{\frac{m(m+1)}{2} }:\psi_{i}(\mathbf{x}_t,\boldsymbol{\Pi}_t)\ge 0\}, 
\end{split}
\end{equation}
where $i \in \{0,...,m-1\}.$
Since $a_{i,t}$ is a DHOCBF with relative degree $m+1-i$ for \eqref{eq:virtual-system}, based on \eqref{eq:highest-SHOCBF}, we define a constraint set $\mathcal{U}_{\boldsymbol{a}}$ for $\boldsymbol{\nu}_t$ as follows: 
\begin{small}
\begin{equation}
\label{eq:constraint-up}
\begin{split}
\mathcal{U}_{\boldsymbol{a}}(\boldsymbol{\Pi}_t)\coloneqq \{\boldsymbol{\nu}_t\in\mathbb{R}^{m}:  \sum_{k=0}^{j_{m}}[(-1)^{j_{m}-k}\binom{j_{m}}{j_{m}-k}\varphi_{i,0}(\boldsymbol{\pi}_{i,t+k})]+\\
\sum_{n=1}^{j_{m}}[\sum_{k=0}^{j_{m}-n}[
(-1)^{j_{m}-n-k}\binom{j_{m}-n}{k}\alpha_{i,n}(\varphi_{i,n-1}(\boldsymbol{\pi}_{i, t+k}))]],\\ i\in \{1,\dots,m\}\},
\end{split}
\end{equation}
\end{small}
where $j_{m}=m+1-i, \varphi_{i,n-1}(\cdot)$ is defined similar to \eqref{eq:virtual-DHOCBFs} and $a_{i, t}$ is ensured positive. $\epsilon$ is a positive constant which can be infinitely small. 

\begin{definition}[DAVCBF]
\label{def:DAVCBF}
Let $\psi_{i}(\mathbf{x},\boldsymbol{\Pi}),\ i\in \{1,...,m\}$ be defined by \eqref{eq:DAVCBF-sequence} and $\mathcal C_{i},\ i\in \{0,...,m-1\}$ be defined by \eqref{eq:DAVCBF-set}. A function $h(\mathbf{x}):\mathbb{R}^{n}\to\mathbb{R}$ is a Discrete-Time Auxiliary-Variable Adaptive Control Barrier Function (DAVCBF) with relative degree $m$ for system \eqref{eq:discrete-dynamics} if every $a_{i,t},i\in \{1,...,m\}$ is a
DHOCBF with relative degree $m+1-i$ for the auxiliary system
\eqref{eq:virtual-system}, and there exist class $\kappa$ functions $\alpha_{j},j\in \{1,...,m\}$  and Lipschitz controller $\mathbf{u}_{t}$ s.t.
\begin{small}
\begin{equation}
\label{eq:highest-DAVCBF}
\begin{split}
\sup_{\mathbf{u}_t\in \mathcal{U},\boldsymbol{\nu}_t\in \mathcal{U}_{\boldsymbol{a}}}[\sum_{k=0}^{m}[(-1)^{m-k}\binom{m}{m-k}\psi_{0}(\mathbf{x}_{t+k},\boldsymbol{\Pi}_{t+k})]+
\sum_{n=1}^{m}[\sum_{k=0}^{m-n}[\\
(-1)^{m-n-k}\binom{m-n}{k}\alpha_{n}(\psi_{n-1}(\mathbf{x}_{t+k},\boldsymbol{\Pi}_{t+k}))]]] \ge 0,
\end{split}
\end{equation}
\end{small}
$\alpha_{j}(\psi_{j-1}(\mathbf{x}_{t},\boldsymbol{\Pi}_{t}))\le \psi_{j-1}(\mathbf{x}_{t},\boldsymbol{\Pi}_{t})$, $\forall (\mathbf{x}_{t},\boldsymbol{\Pi}_{t})\in \mathcal C_{0}\cap,...,\cap \mathcal C_{m-1}$ and each $a_{i,t}>0, i\in\{1,\dots,m\}$. Note that the input $\mathbf{u}_t$ and auxiliary input $\boldsymbol{\nu}_t$ will appear in $\psi_{0}(\mathbf{x}_{t+m},\boldsymbol{\Pi}_{t+m})$ and Eqn. \eqref{eq:highest-DAVCBF} is equivalent to $\psi_{m}(\mathbf{x}_t, \boldsymbol{\Pi}_t)\ge0$.
\end{definition}
\begin{theorem}
\label{thm:safety-guarantee-3}
Given a DAVCBF $h(\mathbf{x})$ from Def. \ref{def:DAVCBF} with corresponding sets $\mathcal{C}_{0}, \dots,\mathcal {C}_{m-1}$ defined by \eqref{eq:DAVCBF-set}, if $(\mathbf{x}_0,\boldsymbol{\Pi}_0) \in \mathcal {C}_{0}\cap \dots \cap \mathcal {C}_{m-1},$ then if there exists a Lipschitz continuous controller $(\mathbf{u}_t,\boldsymbol{\nu}_t)$ that satisfies the constraint in \eqref{eq:highest-DAVCBF} and also ensures $(\mathbf{x}_t,\boldsymbol{\Pi}_t)\in \mathcal {C}_{m-1}$ for all $t \in \mathbb{N}_{\geq 0}$, then $\mathcal {C}_{0}\cap \dots \cap \mathcal {C}_{m-1}$ will be rendered forward invariant for system \eqref{eq:discrete-dynamics}, $i.e., (\mathbf{x}_t,\boldsymbol{\Pi}_t) \in \mathcal {C}_{0}\cap \dots \cap \mathcal {C}_{m-1}, \forall t \in \mathbb{N}_{\geq 0}$. Moreover, $h(\mathbf{x}_t)\ge 0$ is ensured for all $t \in \mathbb{N}_{\geq 0}$
\end{theorem}

\begin{proof}
If $h(\mathbf{x})$ is a DAVCBF, then satisfying the constraint in \eqref{eq:highest-DAVCBF} while ensuring $(\mathbf{x}_t,\boldsymbol{\Pi}_t)\in \mathcal {C}_{m-1}$ for all $t \in \mathbb{N}_{\geq 0}$ is equivalent to making $\psi_{m-1}(\mathbf{x}_t,\boldsymbol{\Pi}_t)\ge 0, \forall t \in \mathbb{N}_{\geq 0}$. Since $a_{m,t}>0$, we have $\frac{\psi_{m-1}(\mathbf{x}_t,\boldsymbol{\Pi}_t)}{a_{m,t}}\ge 0.$ Based on
\eqref{eq:DAVCBF-sequence}, since $(\mathbf{x}_0,\boldsymbol{\Pi}_0) \in \mathcal {C}_{m-2}$ (i.e., $\frac{\psi_{m-2}(\mathbf{x}_0,\boldsymbol{\Pi}_0)}{a_{m-1,0}}\ge 0),a_{m-1,t}>0,$ then we have $\psi_{m-2}(\mathbf{x}_t,\boldsymbol{\Pi}_t)\ge 0$ (This proof is similar to the proof in Rem. \ref{rem:safety-guarantee-2}), and also $\frac{\psi_{m-2}(\mathbf{x}_t,\boldsymbol{\Pi}_t)}{a_{m-1,t}}\ge 0.$ Based on \eqref{eq:DAVCBF-sequence}, since $(\mathbf{x}_0,\boldsymbol{\Pi}_0) \in \mathcal {C}_{m-3},a_{m-2,t}>0$ then similarly we have $\psi_{m-3}(\mathbf{x}_t,\boldsymbol{\Pi}_t)\ge 0$ and $\frac{\psi_{m-3}(\mathbf{x}_t,\boldsymbol{\Pi}_t)}{a_{m-2,t}}\ge 0,\forall t \in \mathbb{N}_{\geq 0}$. Repeatedly, we have $\psi_{0}(\mathbf{x}_t,\boldsymbol{\Pi}_t)\ge 0$ and $\frac{\psi_{0}(\mathbf{x}_t,\boldsymbol{\Pi}_t)}{a_{1,t}}\ge 0,\forall t \in \mathbb{N}_{\geq 0}$. Therefore the intersection of sets $\mathcal {C}_{0},\dots,\mathcal {C}_{m-1}$ is forward invariant and $h(\mathbf{x}_t)=\frac{\psi_{0}(\mathbf{x}_t,\boldsymbol{\Pi}_t)}{a_{1,t}}\ge 0$ is ensured for all $t \in \mathbb{N}_{\geq 0}$.
\end{proof}
Based on Thm. \ref{thm:safety-guarantee-3}, the safety condition $h(\mathbf{x}_t) = \frac{\psi_{0}(\mathbf{x}_t,\boldsymbol{\Pi}_t))}{a_{1,t}} \geq 0$ is guaranteed. The auxiliary inputs $\boldsymbol{\nu}_t$ act as relaxation variables for the constraint \eqref{eq:highest-DAVCBF} and are unbounded, thereby improving the feasibility of the optimization problem.

\subsection{Risk Sensitive Forward Invariance}
In the presence of stochastic uncertainty $\mathbf{w}_t$, ensuring absolute invariance or safety is probably not feasible. Moreover, the uncertainties propagate dynamically over multiple steps until any component of the control input $\mathbf{u}_{t}$ explicitly appears, as shown in Def. \ref{def:relative-degree} and Eqn. \eqref{eq:disturbed system}. In this work, we propose ensuring the forward invariance of sets defined by DAVCBFs in the dynamic coherent risk measure sense to guarantee stochastic safety.
\begin{definition}[$\rho$-Forward Invariance]
\label{def:rho-Forward Invariance}
Given some sets $\mathcal C_{i}$, $i \in \{0,...,m-1\}$ defined by \eqref{eq:DAVCBF-set}, and a time-consistent, dynamic coherent risk measure $\rho_{0:t}$ expressed in \eqref{eq:rho-measure}, we refer to the solutions of \eqref{eq:disturbed system}, initialized at $\mathbf{x}_0 \in \mathcal C_{i}$, as $\rho$-forward invariance if and only if 
\begin{equation}
\label{eq:rho-Forward Invariance}
    \rho_{0:t} (0, 0, \dots, \psi_{i}(\mathbf{x}_t,\boldsymbol{\Pi}_t)) \geq 0, \quad \forall t \in \mathbb{N}_{\geq 0}.
\end{equation}
\end{definition}
Since $\rho_{0:t}$ generally represents a method for measuring risk, Eqn. \eqref{eq:rho-Forward Invariance} implies risk-measured forward invariance. Moreover, $\rho_{0:t}$ can be expressed in a risk-sensitive form, such as CVaR, as shown in \eqref{eq:better CVaR}. Therefore, Eqn. \eqref{eq:rho-Forward Invariance} can also imply risk-sensitive forward invariance.
\subsection{Ensuring Risk Sensitive Safety with RACBFs}
\begin{definition}[RACBF]
\label{def:risk-aware adaptive cbf}
Let $\psi_{i}(\mathbf{x},\boldsymbol{\Pi}),\ i\in \{1,...,m\}$ be defined by \eqref{eq:DAVCBF-sequence}, $\mathcal C_{i},\ i\in \{0,...,m-1\}$ be defined by \eqref{eq:DAVCBF-set} and a dynamic coherent risk measure $\rho$ be defined by Def. \ref{def:Coherent Risk Measure}. A function $h(\mathbf{x}):\mathbb{R}^{n}\to\mathbb{R}$ is a Risk-Aware Adaptive Control Barrier Function (RACBF) with relative degree $m$ for system \eqref{eq:disturbed system} if every $a_{i,t},i\in \{1,...,m\}$ is a
DHOCBF with relative degree $m+1-i$ for the auxiliary system
\eqref{eq:virtual-system}, and there exist class $\kappa$ functions $\alpha_{j},j\in \{1,...,m\}$ and Lipschitz controller $\mathbf{u}_{t}$ for $\psi_{i}(\mathbf{x}, \boldsymbol{\Pi})$ s.t.
\begin{equation}
\label{eq: RACBFs}
    \rho(\psi_{m}(\mathbf{x}_t, \boldsymbol{\Pi}_t))\ge 0,
\end{equation}
$\alpha_{j}(\psi_{j-1}(\mathbf{x}_{t},\boldsymbol{\Pi}_{t}))\le \psi_{j-1}(\mathbf{x}_{t},\boldsymbol{\Pi}_{t})$, 
 for all $(\mathbf{x}_{t},\boldsymbol{\Pi}_{t})$ belonging to the $\rho$-forward invariance of the intersection of sets $ \mathcal C_{0},..., \mathcal C_{m-1}$ and each $a_{i,t}>0, i\in\{1,\dots,m\}$. 
\end{definition}
\begin{theorem}
\label{thm:safety-guarantee-4}
Given a RACBF $h(\mathbf{x})$ from Def. \ref{def:risk-aware adaptive cbf} with corresponding sets $\mathcal{C}_{0}, \dots,\mathcal {C}_{m-1}$ defined by \eqref{eq:DAVCBF-set}, if $(\mathbf{x}_0,\boldsymbol{\Pi}_0) \in \mathcal {C}_{0}\cap \dots \cap \mathcal {C}_{m-1},$ then if there exists solution of Lipschitz controller $(\mathbf{u}_t,\boldsymbol{\nu}_t)$ that satisfies the constraint in \eqref{eq:rho-Forward Invariance} and also ensures $(\mathbf{x}_t,\boldsymbol{\Pi}_t)
$ belongs to $\rho$-forward invariance of $\mathcal {C}_{m-1}$ for all $t\ge 0,$ then the intersection of the sets $ \mathcal C_{0},..., \mathcal C_{m-1}$ will be rendered $\rho$-forward invariance for system \eqref{eq:disturbed system}. Moreover, $\rho(h(\mathbf{x}_t))\ge 0$ is ensured for all $t \in \mathbb{N}_{\geq 0}$.
\end{theorem}
\begin{proof}
If $h(\mathbf{x})$ is a RACBF, then satisfying constraint in \eqref{eq:rho-Forward Invariance} while ensuring $(\mathbf{x}_t,\boldsymbol{\Pi}_t)$ belonging to $\rho$-forward invariance of $ \mathcal {C}_{m-1}$ for all $t \in \mathbb{N}_{\geq 0}$ is equivalent to make $\rho(\psi_{m-1}(\mathbf{x}_t,\boldsymbol{\Pi}_t))\ge 0, \forall t \in \mathbb{N}_{\geq 0}$. Since $a_{m,t}>0$, based on \eqref{eq:DAVCBF-sequence}, we have 
\begin{equation}
\begin{split}
    \rho(\frac{\psi_{m-1}(\mathbf{x}_t,\boldsymbol{\Pi}_t)}{a_{m,t}})=\rho(\bigtriangleup \psi_{m-2}(\mathbf{x}_t,\boldsymbol{\Pi}_t)+\\\alpha_{m-1}(\psi_{m-2}(\mathbf{x}_t,\boldsymbol{\Pi}_t))\ge 0.
\end{split}
\end{equation}
 Based on the properties in Def. \ref{def:Coherent Risk Measure} and $(\mathbf{x}_0,\boldsymbol{\Pi}_0) \in \mathcal {C}_{m-2}$, we have 
 \begin{equation}
\begin{split}
    \rho(\psi_{m-2}(\mathbf{x}_{t+1},\boldsymbol{\Pi}_{t+1}))\ge \rho(\psi_{m-2}(\mathbf{x}_{t},\boldsymbol{\Pi}_{t})-\\
    \alpha_{m-1}(\psi_{m-2}(\mathbf{x}_t,\boldsymbol{\Pi}_t)))\ge \rho(0)=0, \forall t \in \mathbb{N}_{\geq 0}.
\end{split}
\end{equation}
  Since $a_{m-1,t}>0$, we also have $\rho(\frac{\psi_{m-2}(\mathbf{x}_t,\boldsymbol{\Pi}_t)}{a_{m-1,t}})\ge 0.$ Based on \eqref{eq:DAVCBF-sequence}, the properties in Def. \ref{def:Coherent Risk Measure} and $(\mathbf{x}_0,\boldsymbol{\Pi}_0) \in \mathcal {C}_{m-3}$, then similarly we have $\rho(\psi_{m-3}(\mathbf{x}_t,\boldsymbol{\Pi}_t))\ge 0$ and $\rho(\frac{\psi_{m-3}(\mathbf{x}_t,\boldsymbol{\Pi}_t)}{a_{m-2,t}})\ge 0,\forall t \in \mathbb{N}_{\geq 0}$. Repeatedly, we have $\rho(\psi_{0}(\mathbf{x}_t,\boldsymbol{\Pi}_t))\ge 0$ and $\rho(\frac{\psi_{0}(\mathbf{x}_t,\boldsymbol{\Pi}_t)}{a_{1,t}})\ge 0,\forall t \in \mathbb{N}_{\geq 0}$. Therefore the intersection of sets $ \mathcal C_{0},..., \mathcal C_{m-1}$ is $\rho$-forward invariance and $\rho(h(\mathbf{x}_t))=\rho(\frac{\psi_{0}(\mathbf{x}_t,\boldsymbol{\Pi}_t)}{a_{1,t}})\ge 0$ is ensured for all $t \in \mathbb{N}_{\geq 0}$.
\end{proof}
Based on Thm. \ref{thm:safety-guarantee-4}, the risk sensitive safety condition $\rho(h(\mathbf{x}_t))\ge 0$ is enforced.
\subsection{Optimal Control with RACBFs}
Consider the cost \eqref{eq:cost-function-1} in the optimal control Problem \ref{prob:opt-prob}.  Since we need to introduce auxiliary inputs $\nu_{i}$ to enhance the feasibility of optimization, and we aim to incorporate CVaR \eqref{eq:better CVaR} as coherent risk measure to realize risk-aware control, we reformulate the cost in \eqref{eq:cost-function-1} as
\begin{small}
\begin{equation}
\label{eq:cost-function-2}
\begin{split}
J(\mathbf{u}_{t},\mathbf{x}(\mathbf{u}_{t}),\boldsymbol{\nu}_{t},\sigma_{i,t},\zeta_t )=\mathbf{u}_{t}^{T}Q\mathbf{u}_{t}+(\mathbf{x}(\mathbf{u}_{t})-\mathbf{x}_{e})^{T}R(\mathbf{x}(\mathbf{u}_{t})-\mathbf{x}_{e})\\
 +(\boldsymbol{\nu}_{t}-\mathbf{a}_{w})^{T}W(\boldsymbol{\nu}_{t}-\mathbf{a}_{w})+P(\zeta_t+\frac{1}{\beta|\mathcal{W}|}\sum_{i=1}^{|\mathcal{W}|}\sigma_{i,t})  .
\end{split}
\end{equation}
\end{small}
In \eqref{eq:cost-function-2}, $0\le t\le t_T, t \in \mathbb{N}$, $\mathbf{x}_{e}$ denotes a reference state, $\mathbf{a}_{w}$ is the vector to which we hope auxiliary input $\boldsymbol{\nu}_{t}$ converges, and $Q,R,W$ are weight matrices with positive diagonal entries. $P$ is a positive weight scalar for CVaR and the last term of \eqref{eq:cost-function-2} is a transformation of Eqn. \eqref{eq:better CVaR}. To enforce RACBFs in \eqref{eq: RACBFs} with CVaR, we construct reformulated RACBF constraints as follows:
\begin{equation}
\label{eq:reformulated RACBFs}
\begin{split}
-(\zeta_{t}+\frac{1}{\beta|\mathcal{W}|}\sum_{i=1}^{|\mathcal{W}|}\sigma_{i,t})\ge 0,\sigma_{i,t}\ge0, \\
\sigma_{i,t}\ge -\rho^{i}(\psi_{m}(\mathbf{x}_{t}, \boldsymbol{\Pi}_{t}))-\zeta_t, i\in \{1,...,|\mathcal{W}|\}.
\end{split}
\end{equation}
We formulate control bounds $\mathbf{u}_{t}\in \mathcal U \subset \mathbb{R}^{q}$, along with the DHOCBFs and reformulated RACBFs introduced in Eqn. \eqref{eq:constraint-up} and Eqn. \eqref{eq:reformulated RACBFs} as constraints in the nonlinear optimization problem, minimizing the cost function given in Eqn. \eqref{eq:cost-function-2}.

In the auxiliary system \eqref{eq:virtual-system}, if we define $a_{i,t} = \pi_{i,1,t} = 1$ and $\pi_{i,2,t} = \cdots = \pi_{i,m+1-i,t} = 0$, then the construction of functions and sets in \eqref{eq:virtual-DHOCBFs} and \eqref{eq:virtual-sets} aligns exactly with \eqref{eq:high-order-discrete-CBFs} and \eqref{eq:high-order-safety-sets}. This implies that the classical DHOCBF is, in fact, a special case of DAVCBF. Meanwhile, in \eqref{eq:highest-DAVCBF} we see that the $m^{\text{th}}$-order DAVCBF extends $\psi_{0}$ from time step $t$ to time step $t+m$, effectively granting the controller the ability to anticipate the future $m$ steps, enabling it to handle multi-step uncertainties that propagate across multiple stages.
Beyond ensuring safety and feasibility, another advantage of using DAVCBFs is that they help reduce the conservativeness of the control strategy. For example, from \eqref{eq:DAVCBF-sequence}, we can rewrite $\psi_{i}(\mathbf{x}_t,\boldsymbol{\Pi}_t)\ge 0$ as
\begin{equation}
\label{eq:DAVCBF-rewrite}
\begin{split}
\phi_{i-1}(\mathbf{x}_{t+1},\boldsymbol{\Pi}_{t+1})\ge\frac{a_{i,t}}{a_{i,t+1}}(1-\lambda_{i})\phi_{i-1}(\mathbf{x}_{t},\boldsymbol{\Pi}_{t}) ,
\end{split}
\end{equation}
where $\phi_{i-1}(\mathbf{x}_{t},\boldsymbol{\Pi}_{t})=\frac{\psi_{i-1}(\mathbf{x}_{t},\boldsymbol{\Pi}_{t})}{a_{i,t}},\alpha_{i}(\psi_{i-1}(\mathbf{x}_{t},\boldsymbol{\Pi}_{t}))=\lambda_{i}a_{i,t}\phi_{i-1}(\mathbf{x}_{t},\boldsymbol{\Pi}_{t}), 0<\lambda_{i}\le 1, i\in \{1,\dots,m\}$. The term $\frac{a_{i,t}}{a_{i,t+1}}$ can be adjusted to ameliorate the conservativeness of control strategy that $(1-\lambda_{i})\phi_{i-1}(\mathbf{x}_{t},\boldsymbol{\Pi}_{t})$ may have,  which allows for a broader range of feasible control inputs and confirms the adaptivity of DAVCBFs to control constraint and conservativeness of control strategy. The aforementioned benefits can be further amplified by formulating the optimal control problem in the model predictive control framework.
\begin{remark}
\label{rem: sufficient-con}
Note that satisfying the constraint in \eqref{eq:highest-DAVCBF} is a sufficient condition for ensuring the original constraint $\psi_{0}(\mathbf{x}_{t},\boldsymbol{\Pi}_{t}) > 0$. However, it is not necessary to introduce auxiliary variables for all terms from $a_{1,t}$ to $a_{m,t}$. This flexibility allows for selecting an appropriate number of auxiliary variables in the DAVCBF constraints to reduce computational complexity. In other words, the number of auxiliary variables can be less than or equal to the relative degree $m$.
\end{remark}
\begin{remark}
\label{rem: benefits of auxiliary variable}
Another approach to incorporating auxiliary variables into DCBFs is by multiplying a positive $a_{i,t}$ in front of each class $\kappa$ function. Using this method, we rewrite Eqn. \eqref{eq:DAVCBF-rewrite} as:
\begin{equation}
\label{eq:DAVCBF-rewrite-2}
\begin{split}
\phi_{i-1}(\mathbf{x}_{t+1},\boldsymbol{\Pi}_{t+1})\ge (1-\lambda_{i}a_{i,t})\phi_{i-1}(\mathbf{x}_{t},\boldsymbol{\Pi}_{t}).
\end{split}
\end{equation}
Compared to Eqn. \eqref{eq:DAVCBF-rewrite}, we observe that \( (1-\lambda_{i}a_{i,t}) \) has a range of \([- \infty, 1]\), while \( \frac{a_{i,t}}{a_{i,t+1}}(1-\lambda_{i})\phi_{i-1} \) ranges from \([0, +\infty]\). The form of Eqn. \eqref{eq:DAVCBF-rewrite} maximizes the flexibility in adjusting the parameters of the class \( \kappa \) function while eliminating negative parameters that may compromise safety.
\end{remark}
\section{CASE STUDY AND SIMULATIONS}
In this section, we present numerical results to validate our proposed approach using a unicycle robot modeled as a nonlinear discrete-time system. All computations and
simulations were conducted in MATLAB, where the nonlinear optimization problem was solved using IPOPT~\cite{biegler2009large} with the modeling language Yalmip~\cite{lofberg2004yalmip}.
Consider a stochastic discrete-time unicycle model in the form
\begin{equation}
\label{eq:unicycle-model}
\begin{bmatrix} x_{t+1}{-}x_t \\ y_{t+1}{-}y_t \\ \theta_{t+1}{-}\theta_t \\ v_{t+1}{-}v_t \end{bmatrix}{=}\begin{bmatrix} v_{t} \cos(\theta_{t}) \Delta t \\ v_{t}\sin(\theta_{t}) \Delta t \\ 0 \\ 0 \end{bmatrix}{+}\begin{bmatrix} 0 & 0 \\ 0 & 0 \\ \Delta t & 0 \\ 0 & \frac{\Delta t}{M}  \end{bmatrix}
\begin{bmatrix} u_{1,t} \\ u_{2,t} \end{bmatrix}+\mathbf{w}_t,
\end{equation}
where $M$ denotes the mass, and $\mathbf{x}_{t}=[x_{t},y_{t},\theta_{t},v_{t}]^{T}$ captures the 2-D location, heading angle, and linear speed; $\mathbf{u}_{t}=[u_{1,t},u_{2,t}]^{T}$ represents angular velocity ($u_{1,t}$) and driven force ($u_{2,t}$), respectively. The disturbance $\mathbf{w}_t \in \mathcal{W}$ enters the system linearly and is characterized by a probability mass function over the states. In this study, the probability mass function is a simple normal distribution centered at $0$ with a standard deviation of $[0.2, 0.2, 0.2, 0.2]$ for the four states $[x_t, y_t, \theta_t, v_t]$. The system is discretized with $\Delta t = 0.1s$ and the total number of discretization steps $t_T$ in Problem. \ref{prob:opt-prob} is 50. System \eqref{eq:unicycle-model} is subject to the following input constraint:
\begin{equation}
\begin{split}
\label{eq:state-input-constraint}
\mathcal{U}&=\{\mathbf{u}_{t}\in \mathbb{R}^{2}: [-5, -5M]^{T} \le \mathbf{u}_{t}\le [5, 5M]^{T}\}.
\end{split}
\end{equation}
The initial state is $[-1.5,0,\frac{\pi}{12},1]^{T}$ and the target state is $\mathbf{x}_{e}=[2,0,0,0]^{T},$ which are marked as black diamond and green circle in Fig.~\ref{fig:trajectory}. The circular obstacle is centered at $(x_{o},y_{o})=(0,0)$ with $r_o=1$, which is displayed in pale red. To initialize the proposed RACBFs, we multiply an auxiliary variable $a_{1,t}$ by a quadratic distance function for circular obstacle avoidance as $\psi_{0}(\mathbf{x}_{t},\boldsymbol{{\pi}}_{1,t})=a_{1,t}((x_{t}-x_{o})^{2}+(y_{t}-y_{o})^{2}-r_{o}^{2})$. Since the relative degree of $\psi_{0}$ with respect to system \eqref{eq:unicycle-model} is 2, the DHOCBFs for $a_{1,t}$ are defined as 
\begin{equation}
\label{eq:SHOCBF-sequence-unicycle}
\begin{split}
&\varphi_{1,0}(\boldsymbol{{\pi}}_{1,t})\coloneqq a_{1,t},\\
&\varphi_{1,1}(\boldsymbol{{\pi}}_{1,t})\coloneqq \bigtriangleup \varphi_{1,0}(\boldsymbol{{\pi}}_{1,t})+l_{1}\varphi_{1,0}(\boldsymbol{{\pi}}_{1,t}),\\
&\varphi_{1,2}(\boldsymbol{{\pi}}_{1,t})\coloneqq \bigtriangleup \varphi_{1,1}(\boldsymbol{{\pi}}_{1,t})+l_{2}\varphi_{1,1}(\boldsymbol{{\pi}}_{1,t}),
\end{split}
\end{equation}
where $\alpha_{1,1}(\cdot),\alpha_{1,2}(\cdot)$ are defined as linear functions. The other DAVCBFs are then defined as 
\begin{equation}
\label{eq:DAVCBF-sequence-unicycle}
\begin{split}
&\psi_{1}(\mathbf{x}_{t},\boldsymbol{{\pi}}_{1,t})\coloneqq \bigtriangleup \psi_{0}(\mathbf{x}_{t},\boldsymbol{{\pi}}_{1,t})+\lambda_{1}\psi_{0}(\mathbf{x}_{t},\boldsymbol{{\pi}}_{1,t}),\\
&\psi_{2}(\mathbf{x}_{t},\boldsymbol{{\pi}}_{1,t})\coloneqq \bigtriangleup \psi_{1}(\mathbf{x}_{t},\boldsymbol{{\pi}}_{1,t})+\lambda_{2}\psi_{1}(\mathbf{x}_{t},\boldsymbol{{\pi}}_{1,t}),
\end{split}
\end{equation}
where $\alpha_{1}(\cdot),\alpha_{2}(\cdot)$ are set as linear functions. By formulating constraints from DHOCBFs \eqref{eq:SHOCBF-sequence-unicycle}, control bounds \eqref{eq:state-input-constraint}, and reformulated
RACBFs \eqref{eq:reformulated RACBFs} derived from DAVCBFs \eqref{eq:DAVCBF-sequence-unicycle}, we can define cost function as
\begin{small}
\begin{equation}
\label{eq:cost-function-3}
\begin{split}
\min_{\mathbf{u}_{t},\mathbf{x}(\mathbf{u}_{t}),\nu_{1,t},\sigma_{i,t},\zeta_t}[\mathbf{u}_{t}^{T}Q\mathbf{u}_{t}+(\mathbf{x}(\mathbf{u}_{t})-\mathbf{x}_{e})^{T}R(\mathbf{x}(\mathbf{u}_{t})-\mathbf{x}_{e})\\
 +W(\nu_{1,t}-a_{1,w})^{2}+P(\zeta_t+\frac{1}{\beta|\mathcal{W}|}\sum_{i=1}^{|\mathcal{W}|}\sigma_{i,t})].
\end{split}
\end{equation}
\end{small}
Other parameters are set as $a_{1,w}=0, W=1000,P=1$, $Q=\text{diag}(1000, 1000)$, $R = \text{diag}(10000, 10000, 1000, 10)$, $|\mathcal{W}|=20,M=1650,a_{1,0}=0.1, \pi_{1,2,0}=0,\epsilon =10^{-10}$.

We construct two benchmarks based on DHOCBFs: risk-aware DHOCBF and risk-agnostic DHOCBF. For the risk-aware DHOCBF, we remove the auxiliary variables and auxiliary inputs from \eqref{eq:DAVCBF-sequence-unicycle} as well as the third term of the cost function in \eqref{eq:cost-function-3}, while reformulating the DHOCBF using the coherent risk measure from \eqref{eq:reformulated RACBFs}.  For the risk-agnostic DHOCBF, we also remove the auxiliary variables and auxiliary inputs from \eqref{eq:DAVCBF-sequence-unicycle} but do not reformulate the DHOCBF using the coherent risk measure. This means that the risk-agnostic DHOCBF constraints do not account for disturbances. All benchmarks maintain the same parameter settings as the RACBF method.

\begin{figure}[ht]
    \centering
    \includegraphics[scale=0.50]{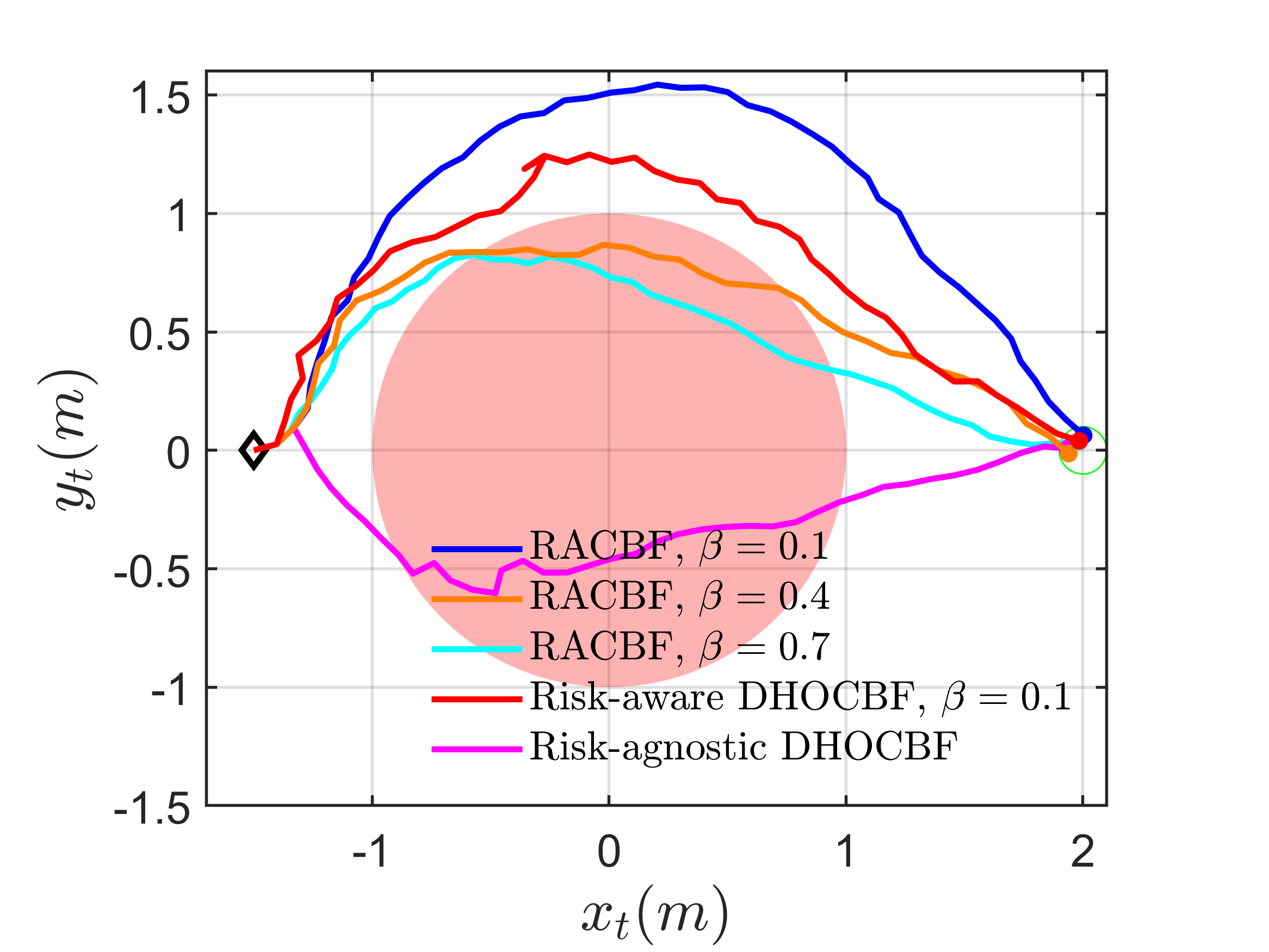}
    \caption{Closed-loop trajectories with controllers RACBF (blue, orange, cyan), risk-aware DHOCBF (red) and risk-agnostic DHOCBF (pink). RACBF and risk-aware DHOCBF with $\beta=0.1$ work well for safety-critical navigation under state disturbances.}
    \label{fig:trajectory}
\end{figure} 

We first analyze safety for different control methods. We set hyperparameters as $\lambda_1=\lambda_2=l_1=l_2=0.4$ for all methods, with $\beta=0.1$ for the risk-aware DHOCBF, while varying $\beta$ for RACBF as shown in Fig. \ref{fig:trajectory}. In the simulation, we observed that only RACBF ensures that the optimization problem remains feasible at all times. In contrast, under DHOCBF-based methods, the system fails to reach the green circle (radius 0.1) due to infeasibility in solving the optimization problem. To address this, we expanded the control input bounds for DHOCBF-based methods, allowing them to generate complete closed-loop trajectories (for more details on feasibility, see the discussion on Fig. \ref{fig:input}). By analyzing the closed-loop trajectories from different controllers, we conclude that reducing the significance risk level $\beta$ in CVaR helps ensure safety under state disturbances (see blue and red curves). The absence of a risk measure or an excessively high significance risk level leads to the system entering the unsafe set (see orange, cyan and pink curves). A smaller $\beta$ means CVaR captures a smaller portion of the loss distribution, making it more focused on extreme worst-case scenarios. Adjusting $\beta$ directly reduces the conservatism of risk assessment when $\beta$ is raised.
\begin{figure}[ht]
    \centering
    \includegraphics[scale=0.50]{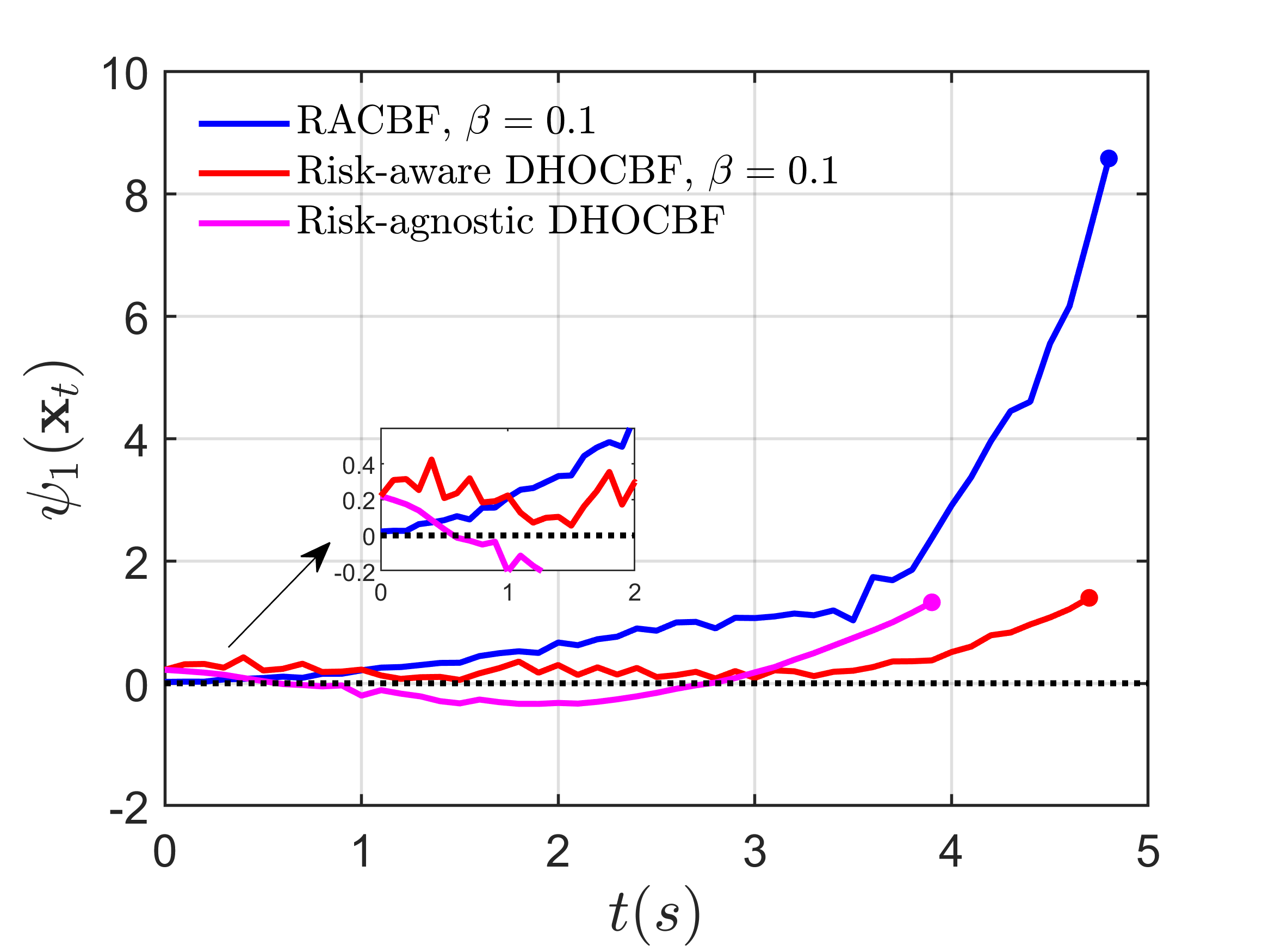}
    \caption{First-order functions over time with controllers RACBF (blue), risk-aware DHOCBF (red) and risk-agnostic DHOCBF (pink). RACBF and risk-aware DHOCBF with $\beta=0.1$ effectively maintain the first-order function non-negative under state disturbances.}
    \label{fig:firstorder}
\end{figure} 

Next, we analyze how RACBF and risk-aware DHOCBF effectively handle multi-step uncertainties that propagate over the relative degree. Since the relative degree of RACBF and risk-aware DHOCBF candidates with respect to system \eqref{eq:unicycle-model} is 2, the control input fully appears in the second-order constraint, introducing a two-time-step delay. In the constraint at each order, state disturbances enter the system, potentially leading to unsafe situations.
The high-order coherent risk measure formulation in RACBF and risk-aware DHOCBF ensures that lower-order functions (e.g., the first-order function $\psi_1$ in Fig. \ref{fig:firstorder}) remain non-negative, thereby maintaining $h$ non-negative and ensuring safety. In contrast, risk-agnostic DHOCBF fails to guarantee the non-negativity of the first-order function, leading to unsafe behavior.
\begin{figure}[ht]
    \centering
    \includegraphics[scale=0.50]{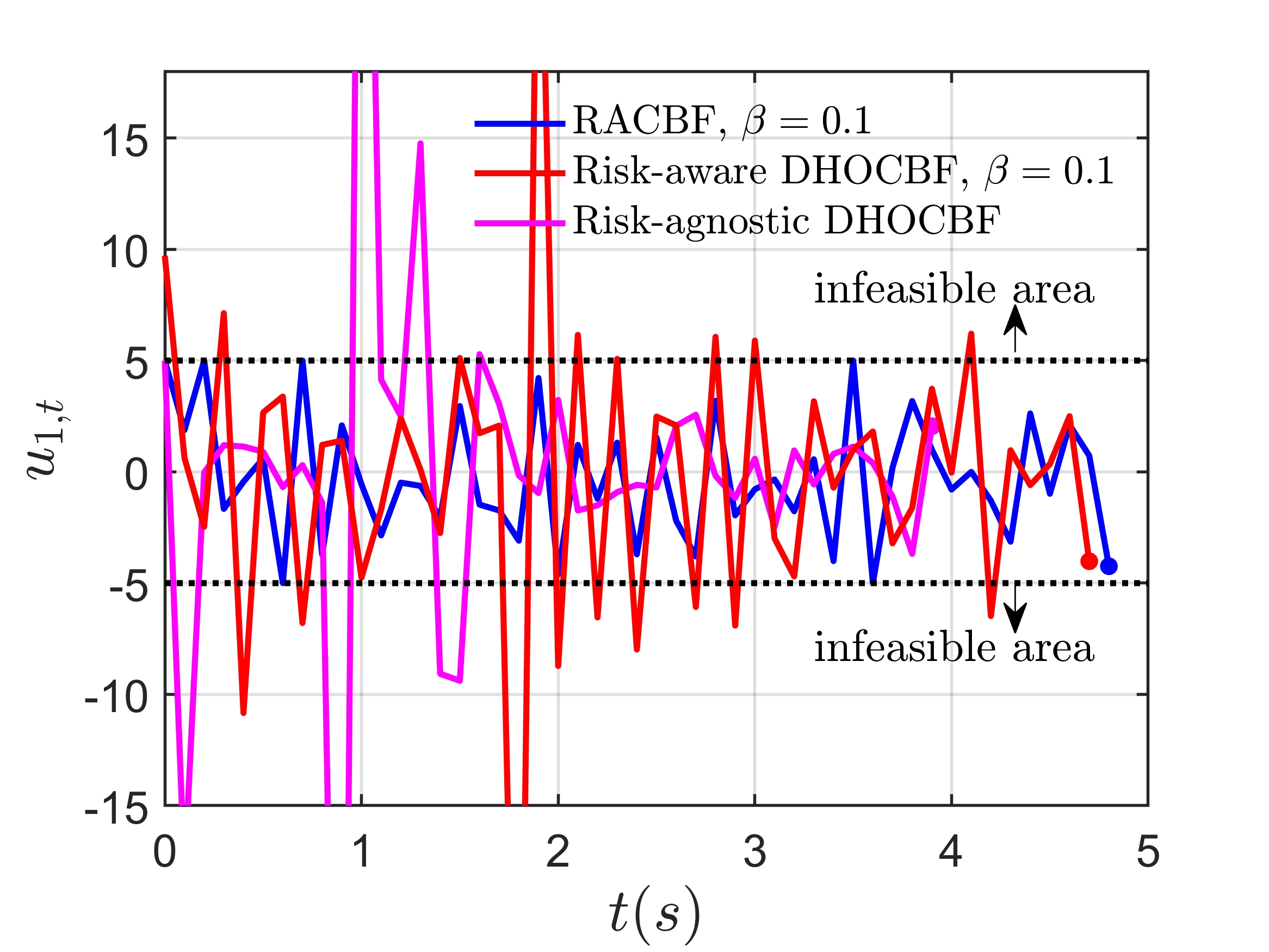}
    \caption{Control input $u_{1,t}$ (angular velocity) over time with controllers RACBF (blue), risk-aware DHOCBF (red) and risk-agnostic DHOCBF (pink). Only RACBF with $\beta=0.1$ effectively keeps $u_{1,t}$ within the corresponding bounds under state disturbances, ensuring feasibility when strict bounds are imposed. In contrast, for the other two methods, $u_{1,t}$ exceeds the corresponding bounds, making the optimization problem infeasible under strict input bounds.}
    \label{fig:input}
\end{figure} 

Finally, we analyze the feasibility of different control methods. As discussed in Fig. \ref{fig:trajectory}, we expanded the control input bounds for DHOCBF-based methods. Otherwise, the control inputs obtained from risk-aware DHOCBF and risk-agnostic DHOCBF would conflict with the narrow input bounds, making the optimization problem infeasible, as shown in Fig. \ref{fig:input}. Expanding the input bounds allows the optimized input to exceed the narrow bounds while keeping the optimization problem feasible, thereby generating a complete closed-loop trajectory. In contrast, RACBF consistently keeps the control input within the narrow input bounds while still producing a safe trajectory. This demonstrates RACBF's adaptivity to input constraints and its capability to enhance feasibility.
\section{Conclusion and Future Work}
\label{sec:Conclusion and Future Work}
In this work, we introduced Risk-Aware Adaptive CBFs (RACBFs) as a novel approach to enforcing safety constraints in stochastic nonlinear systems. Our method integrates Discrete-Time Auxiliary-Variable Adaptive CBFs (DAVCBFs) to improve feasibility and incorporates coherent risk measures to enable risk-sensitive safety enforcement. We showed that RACBFs extend discrete-time high-order CBF constraints over multiple time steps, effectively addressing multi-step uncertainties that propagate through system dynamics. Simulation results on a stochastic unicycle system demonstrated that RACBFs outperform standard robust formulations of discrete-time CBF methods by maintaining feasibility while reducing conservatism. The ability of RACBFs to dynamically adapt to risk levels provides a promising direction for future research in safe and efficient stochastic control, such as finite-time reachability problems in real-world autonomous systems.
\bibliographystyle{IEEEtran}
\balance
\bibliography{references.bib}
\end{document}